\documentclass[10pt]{article}
\usepackage{amsmath,amssymb,amsthm,pb-diagram,lamsarrow,pb-lams, hyperref, euscript, ulem, mathcomp}
\usepackage[matrix,arrow,curve]{xy}
\usepackage{graphicx}
\usepackage{tabularx}
\usepackage{float}
\usepackage{hyperref}

\usepackage[T1]{fontenc}

\usepackage[sc]{mathpazo}
\usepackage[usenames]{color}
\usepackage{colortbl}

\DeclareFontFamily{T1}{pzc}{}
\DeclareFontShape{T1}{pzc}{m}{it}{1.8 <-> pzcmi8t}{}
\DeclareMathAlphabet{\mathpzc}{T1}{pzc}{m}{it}

\title{Noncommutative covering projections and $K$-homology }

\theoremstyle{plain}
\newtheorem{prop}{Proposition}[section]

\newtheorem{lem}[prop]{Lemma}
\newtheorem{thm}[prop]{Theorem}

\theoremstyle{definition}
\newtheorem{defn}[prop]{Definition}
\newtheorem{empt}[prop]{}
\newtheorem{exm}[prop]{Example}
\newtheorem{rem}[prop]{Remark}

\theoremstyle{remark}

\chardef\bslash=`\\ 

\newcommand{\K}{\mathcal{K}}

\newcommand{\Rb}{\mathbb{R}}
\newcommand{\Zb}{\mathbb{Z}}

\newcommand{\eps}{\varepsilon}

\newcommand{\rar}{\rightarrow}

\newbox\ncintdbox \newbox\ncinttbox 
\setbox0=\hbox{$-$} \setbox2=\hbox{$\displaystyle\int$}
\setbox\ncintdbox=\hbox{\rlap{\hbox
    to \wd2{\hskip-.125em \box2\relax\hfil}}\box0\kern.1em}
\setbox0=\hbox{$\vcenter{\hrule width 4pt}$}
\setbox2=\hbox{$\textstyle\int$} \setbox\ncinttbox=\hbox{\rlap{\hbox
    to \wd2{\hskip-.175em \box2\relax\hfil}}\box0\kern.1em}

\begin{document}
\maketitle  \setlength{\parindent}{0pt}
\begin{center}
\author{}
{\textbf{Petr R. Ivankov*}\\
e-mail: * monster.ivankov@gmail.com \\
}
\end{center}

\vspace{1 in}

\begin{abstract}
\noindent

If $X$ is a topological space then there is a natural homomorphism $\pi_1(X)\rightarrow K_1(X)$ from a fundamental group to a $K_1$-homology group. Covering projections depend of fundamental group.  So $K_1$-homology groups are interrelated with covering projections. This article is concerned with a noncommutative analogue of this interrelationship.

\end{abstract}
\tableofcontents

\section{Introduction}
It is known that $K_1(S^1)\approx \mathbb{Z}$. If $x$ is a generator of $K(S^1)$ than there is a natural homomorphism $\varphi_{K}:\pi_1(X)\rightarrow K_1(X)$ given by
\begin{equation}\label{pi1_to_K1_hom}
[f] \mapsto K_1(f)(x)
\end{equation}
where $f$ is a representative of $[f]\in \pi_1(X)$. This homomorphism does not depend on a basepoint because $K_1(X)$ is an abelian group . So the basepoint is omitted. Let $K_{11}(X)\subset K_1(X)$ be the image of $\varphi_K$. Then $K_{11}(X)$ is a homotopical invariant.
\begin{exm}\label{circle_cov}
We have a natural isomorphism $\varphi_{K}:\pi_1(S^1) \rar K_1(S^1)$. From $\pi_1(S^1)=\mathbb{Z}$ it follows that there is a $n$-listed covering projection $f_n:S^1 \rightarrow S^1$ for any $n \in \mathbb{N}$.
\end{exm}.
\begin{exm}\label{n_cirle_cov}
Let  $f: S^1 \rightarrow S^1$ be an $n$ listed covering projection, $C_{f}$ is the mapping cone \cite{spanier:at} of $f$. Then $\pi_1(C_{f})\approx K_1(C_f) \approx\mathbb{Z}_n$ and there is a natural isomorphism $\varphi_{K}: \pi_1(C_{f}) \rar K_1(C_{f})$. There is $n$ - listed universal covering projection $f_n: \widehat{C_{f}}\rightarrow C_{f}$.
\end{exm}\label{n_cirle_ex}
Finitely listed covering projections depend of fundamental group. Any epimorphism $\pi_1(X) \rightarrow \mathbb{Z}$ (resp. $\pi_1(X) \rightarrow \mathbb{Z}_n$) corresponds to the infinite sequence of finitely listed covering projections (resp. an $n$ - listed covering projection). If $\varphi : \pi_1(X) \rightarrow G$ is an epimorphism ($G\approx \mathbb{Z}$ or $G\approx \mathbb{Z}_n$) such that $ \mathrm{ker} \ \varphi_K\subset \mathrm{ker} \ \varphi$ then there is an algebraic construction of these covering projections which is described in this article. A noncommutative analogue of $K_{11}(X)$ is discussed.

This article assumes elementary knowledge of following subjects
\begin{enumerate}
\item Algebraic topology  \cite{spanier:at}.
\item $C^*-$ algebras and $K$-theory \cite{blackadar:ko}, \cite{dixmier_a_r}, \cite{murphy},  \cite{pedersen:ca_aut}.  

\end{enumerate}

Following notation is used.
\newline
\begin{tabular}{|c|c|}
\hline
Symbol & Meaning\\
\hline
$A^+$  & Unitization of $C^*-$ algebra $A$\\
$A_+$  & A positive cone of $C^*-$ algebra $A$\\
$A^G$  & Algebra of $G$ invariants, i.e. $A^G = \{a\in A \ | \ ga=a, \forall g\in G\}$\\
$\hat A$ & Spectrum of  $C^*$ - algebra $A$  with the hull-kernel topology \\
 & (or Jacobson topology)\\
$\mathrm{Aut}(A)$ & Group * - automorphisms of $C^*$  algebra $A$\\
$B(H)$ & Algebra of bounded operators on Hilbert space $H$\\
$B_{\infty}=B_{\infty}(\{z\in \mathbb{C} \ | \ |z|=1\})$  & Algebra of Borel measured functions on the $\{z\in \mathbb{C} \ | \ |z|=1\}$ set. \\
$\mathbb{C}$ (resp. $\mathbb{R}$)  & Field of complex (resp. real) numbers \\
$\mathbb{C}^*$ & $\{z \in \mathbb{C} \ | \ |z| = 1\}$ \\
$C(X)$ & $C^*$ - algebra of continuous complex valued \\
 & functions on topological space $X$\\
$C^b(X)$ & $C^*$ - algebra of bounded  continuous complex valued \\
$H$ &Hilbert space \\
$I = [0, 1] \subset \mathbb{R}$ & Closed unit  interval\\
$G_{tors} \subset G$  & The torsion subgroup of an abelian group\\
$\mathcal{K}(H)$ or $\mathcal{K}$ & Algebra of compact operators on Hilbert space $H$\\
$\mathbb{M}_n(A)$  & The $n \times n$ matrix algebra over $C^*-$ algebra $A$\\

$\mathrm{Map}(X, Y)$  & The set of maps from $X$ to $Y$\\
$M(A)$  & A multiplier algebra of $C^*$-algebra $A$\\
$M^s(A) = M(A \otimes \mathcal{K})$  & Stable multiplier algebra of $C^*-$ algebra $A$\\
$\mathbb{N}$ & Monoid of natural numbers \\
$Q(A)=M(A)/A$  & Outer multiplier algebra of $C^*-$ algebra $A$\\

$Q^s(A)=(M(A \otimes \mathcal{K}))/(A  \otimes \mathcal{K})$  & Stable outer multiplier algebra of $C^*-$ algebra $A$\\

$\mathbb{Q}$  & Field of rational numbers \\
  $\mathrm{sp}(a)$ & Spectrum of element of $C^*$-algebra $a\in A$  \\
$U(H) \subset \mathcal{B}(H) $ & Group of unitary operators on Hilbert space $H$\\
$U(A) \subset A $ & Group of unitary operators of algebra $A$\\

$\mathbb{Z}$ & Ring of integers \\

$\mathbb{Z}_m$ & Ring of integers modulo $m$ \\
$\Omega$ &  Natural contravariant functor from category  of commutative \\ & $C^*$ - algebras, to category of Hausdorff spaces\\\hline
\end{tabular}
\newline
\newline

\section{Galois extensions of $C^*$ - algebras and noncommutative covering projections}
\subsection{General theory}
\begin{empt}{\it Galois extensions}.
Let $G$ be a finite group, a $G$-Galois extensions can be regarded as particular case of Hopf-Galois extensions \cite{hajac:toknotes}, where Hopf algebra is a commutative algebra $C(G)$. Let $A$ be a $C^*$-algebra, let $G\subset \mathrm{Aut}(A)$ be a finite group of $*$- automorphisms. Let $_A\mathcal{M}^G$ be a category of $G$-equivariant modules. There is a pair of adjoint functors $(F,U)$ given by
\begin{equation}\label{f_functor}
F = A \otimes_{A^{G}} -: _{A^G}M \rightarrow _A\mathcal{M}^G;
\end{equation}
\begin{equation}\label{u_functor}
U = (-)^G: _A\mathcal{M}^G \rightarrow _{A^G}\mathcal{M}. 
\end{equation}

The unit and counit of the adjunction $(F, U)$ are given by the formulas
\begin{equation}\nonumber
\eta_{N} : N \rightarrow (A \otimes_{A^{G}} N)^{G}, \ \eta_{N}(n) = 1 \otimes  n;
\end{equation}
\begin{equation}\nonumber
\eps_{M} : A \otimes_{A^{G}} M^{G} \rightarrow M, \ \eps_{M}(a\otimes m) = am.
\end{equation}

Consider a following map
\begin{equation}\label{can_def}
\mathrm{can}: A \otimes_{A^G} A \rightarrow \mathrm{Map}(G, A)
\end{equation}
given by
\begin{equation}\nonumber
a_1 \otimes a_2 \mapsto (g \mapsto a_1 (ga_2)), \ (a_1, a_2 \in A, \ g \in G).
\end{equation}
The $\mathrm{can}$ is a $_A\mathcal{M}^G$ morphism.
\end{empt}
\begin{thm}\label{hoph_galois_def_thm}\cite{morita_hopf_galois}
Let $A$ be an algebra, let $G$ be a finite group which acts on $A$, $(F, U)$ functors given by (\ref{f_functor}), (\ref{u_functor}). Consider the following
statements:
\begin{enumerate}
\item $(F,U)$ is a pair of inverse equivalences;
\item  $(F,U)$ is a pair of inverse equivalences and  $A \in _{A^{G}} \mathcal{M}$ is flat;
\item The  $\mathrm{can}$ is an isomorphism and $A \in_{A^{G}} \mathcal{M}$ is faithfully flat.

\end{enumerate}
These the three conditions are equivalent.
\end{thm}

\begin{defn}\label{hoph_galois_def}
If conditions of  theorem \ref{hoph_galois_def_thm} are hold, then $A$ is said to be {\it left
faithfully flat $G$-Galois extension}
\end{defn}
\begin{rem}
Theorem \ref{hoph_galois_def_thm} is an adapted to finite groups version of theorem from \cite{morita_hopf_galois}.
\end{rem}

In case of commutative $C^*$-algebras definition \ref{hoph_galois_def} supplies finitely listed covering projections of topological spaces. However I think that above definition is not quite good analogue of noncommutative covering projections. Noncommutative algebras contains inner automorphisms. Inner automorphisms are rather gauge transformations \cite{gross_gauge} than geometrical ones. So I think that inner automorphisms should be excluded. Importance of outer automorphisms was noted by  Miyashita \cite{miyashita:finite_outer_galois}. It is reasonably take to account outer automorphisms only. I have set more strong condition.   
\begin{defn}\label{gen_in_def}\cite{rieffel_finite_g}
Let  $A$ be $C^*$ - algebra. A *- automorphism $\alpha$ is said to be {\it generalized inner} if is obtained by conjugating with unitaries from multiplier algebra $M(A)$.
\end{defn}
\begin{defn}\label{part_in_def}\cite{rieffel_finite_g}
Let  $A$ be $C^*$ - algebra. A *- automorphism $\alpha$ is said to be {\it partly inner} if its restriction to some non-zero $\alpha$- invariant two-sided ideal is generalized inner. We call automorphism {\it purely outer} if it is not partly inner.
\end{defn}
Instead definitions \ref{gen_in_def}, \ref{part_in_def} following definitions are being used. 
\begin{defn}
Let $\alpha \in \mathrm{Aut}(A)$ be an automorphism. A representation $\rho : A\rightarrow B(H)$ is said to be {\it $\alpha$ - invariant} if a representation $\rho_{\alpha}$ given by
\begin{equation}
\rho_{\alpha}(a)= \rho(\alpha(a))
\end{equation}
is unitary equivalent to $\rho$.
\end{defn}
\begin{defn}
Automorphism $\alpha \in \mathrm{Aut}(A)$ is said to be {\it strictly outer} if for any $\alpha$- invariant representation $\rho: A \rightarrow B(H) $, automorphism $\rho_{\alpha}$ is not a generalized inner automorphism.
\end{defn}
\begin{defn}\label{nc_fin_cov_pr_defn}
Let $A$ be a $C^*$ - algebra and $G \subset \mathrm{Aut}(A)$ be a finite subgroup of * - automorphisms.
An injective * - homomorphism $f : A^G \rar A$ is said to be a {\it noncommutative finite
covering projection} (or {\it noncommutative $G$ - covering projection})  if $f$ satisfies following conditions:
\begin{enumerate}
\item $A$ is a finitely generated equivariant projective left and right $A^G$ Hilbert $C^*$-module.
\item If $\alpha \in G$ then $\alpha$ is strictly outer.
\item $f$ is a left faithfully flat $G$-Galois extension.
\end{enumerate}
The $G$ is said to be {\it covering transformation group} of $f$. Denote by $G(B|A)$ covering transformation group of covering projection $A \rar B$.
\end{defn}

\begin{empt}\label{irreducible_corr}{\it Irreducible representations of noncommutative covering projections}. Let $f:A^G \rar A$ be a noncommutative $G$ -  covering projection. Let $\rho:  A \rar\mathcal{B}(H)$ be an irreducible representation.
Let $g \in G$ and $\rho_g:  A \rar\mathcal{B}(H)$ be such  that
\begin{equation}\nonumber
\rho_g (a)= \rho(ga).
\end{equation}
So it is an action of $G$ on $\hat A$ such that
\begin{equation}\label{action_of_g_on_spectrum}
g \mapsto (\rho \mapsto \rho_g); \ \forall g \in G, \forall \rho \in \hat A.
\end{equation}

Let us enumerate elements of $G$ by integers, i. e. $g_1, ..., g_n \in G , \ n = |G|$ and define action of $\sigma: G \times \{i, ..., n\} \rar \{i, ..., n\}$ such that $\sigma(g, i) = j \Leftrightarrow  g_j = gg_i$
Let $\rho_{\oplus} = \oplus_{g \in G} \rho_g: A\rar B (H^n)$ be such that
\begin{equation}\label{a_act_irr_sum}
\rho_{\oplus} (a)(h_1, ..., h_n)= (\rho(g_1a)h_1, ...,  (\rho(g_na)h_n).
\end{equation}
Let us define such linear action of $G$ on $H^n$ that
\begin{equation}\label{g_act_irr_sum}
g (h_1, ..., h_n)  = (h_{\sigma(g^{-1}, 1)}, ...,h_{\sigma(g^{-1}, n)}).
\end{equation}
From (\ref{a_act_irr_sum}), (\ref{g_act_irr_sum}) it follows that 
\begin{equation}\nonumber
g(ah)=(ga)(gh); \ \forall a\in A, \ \forall g\in G, \ \forall h\in H^n,
\end{equation}
i.e. $H^n \in _A\mathcal{M}^G$. Equivariant representation $\rho_{\oplus}$ defines representation $\eta: A^G \rar B(K)$. $K = \left(H^n\right)^G$. If $\eta$ is not an irreducible then there is a nontrivial
$A^G$ - submodule $N  \varsubsetneq K$. From $ _A\mathcal{M}^G \approx  _{A^G}\mathcal{M}$ it follows that $A \otimes_{A^G} N \varsubsetneq H^n$ is a nontrivial $A$ - submodule. If we identify $H$ with first summand of $H^n$ then $(A \otimes_{A^G} K) \cap H \varsubsetneq  H$ is a nontrivial $A$ - submodule. This fact contradicts with that $\rho$ is irreducible. So $\eta$ is
an irreducible representation. In result we have a natural map 
\begin{equation}\label{spectrum_galois_map}
\hat f : \hat A \rar \widehat{A^G}, \ (\rho \mapsto \eta)
\end{equation}
and
\begin{equation}\label{spectrum_galois_quot}
\widehat{A^G} \approx \hat A / G.
\end{equation}

\end{empt}

\subsection{Covering projection of $C^*$-algebras with continuous trace}
\begin{defn}\label{abelian_element_defn}\cite{pedersen:ca_aut}
A positive element in $C^*$ - algebra $A$ is {\it abelian} if subalgebra $xAx \subset A$ is commutative.
\end{defn}
\begin{prop}\label{abelian_element_proposition}\cite{pedersen:ca_aut}
A positive element $x$ in $C^*$ - algebra $A$ is abelian if $\mathrm{dim} \ \pi(x) \le 1$ for every irreducible representation $\pi: A \rar \mathcal{B}(H)$ of $A$.
\end{prop}

\begin{empt} Let $A$ be a $C^*$ - algebra. For each $x\in A_+$ the (canonical) trace $\mathrm{Tr}(\pi(x))$ of $\pi(x)$ depends only on the equivalence class of an irreducible representation $\pi:A \rightarrow B(H)$, so that we may define a function $\hat x : \hat A \rar [0,\infty]$ by $\hat x(t)=\mathrm{Tr}(\pi(x))$ whenever $\pi\in t$. From Proposition 4.4.9 \cite{pedersen:ca_aut} it follows that $\hat x$ is lower semicontinuous function on a in Jacobson topology.
\end{empt}
\begin{defn}\label{continuous_trace_c_a_defn}\cite{pedersen:ca_aut} We say that element $x\in A_+$ has {\it continuous trace} if $\hat x \in C^b(\hat A)$. We say that $A$ is a $C^*$ - algebra {\it with continuous trace} if set of elements with continuous trace is dense in $A_+$. We say that a $C^*$ - algebra $A$ is of type $I$ if each non-zero quotient of $A$ contains non-zero
abelian element. If $A$ is even generated (as $C^*$ - algebra) by its abelian elements we say
that it is of type $I_0$.
\end{defn}

\begin{thm} (Theorem 5.6 \cite{pedersen:ca_aut}) For each $C^*$ - algebra $A$ there is a dense hereditary ideal $K(A)$,
which is minimal among dense ideals.

\end{thm}

\begin{prop}\label{continuous_trace_c_a_proposition}\cite{pedersen:ca_aut}
Let $A$ be a $C^*$ - algebra with continuous trace Then
\begin{enumerate}
\item $A$ is of type $I_0$;
\item $\hat A$ is a locally compact Hausdorff space;
\item For each $t \in \hat A$ there is an abelian element $x \in A$ such that $\hat x \in K(\hat A)$ and $\hat x(t) = 1$.
\end{enumerate}
The last condition is sufficient for $A$ to have continuous trace.
\end{prop}

\begin{rem}\label{ctr_is_ccr}
From \cite{dixmier_tr}, Proposition 10, II.9 it follows that a continuous trace
$C^*$-algebra is always a $CCR$-algebra, a $C^*$-algebra where for every irreducible
representation $\pi: A \rightarrow B(H)$ and for every element $x\in A$, $\pi (x)$ is a
compact operator, i.e. $\pi(A)=\mathcal{K}(H)$.

\end{rem}
\begin{lem}\label{free_action_ccr}
Let $A^G\rightarrow A$ be  a noncommutative covering projection such that $A$ is a $CCR$-algebra. Then $G$ acts freely on $\hat A $.
\end{lem}
\begin{proof}

Suppose that $G$ does not act freely on $\hat A$.  Then there are $x \in \hat A$ and $g \in G$ such that $gt = t \ (t\in \hat A)$. By definition \ref{nc_fin_cov_pr_defn} $g$ should be strictly outer.  Let $\rho: A \rar B(H)$ be representative of $x$. Then $\rho_g$ is also representative of $x$. So $\rho$ is unitary equivalent to $\rho_g$, i. e. there is unitary $U \in U(H)$ such that $\rho_g(a) = U\rho(a)U^*$ ($\forall a \in A)$. According to \ref{ctr_is_ccr} $\rho(A) = \mathcal{K}(H)$, $\rho(M(A)) = B(H)$, $\rho(U(M(A))) = U(H)$. So it is $u\in M(A)$ such that $\rho(u)=U$ and we have $\rho_g(a) = \rho(u)\rho(a)\rho(u^*)$. It means that $g$ is inner with respect to $\rho$, so action of $g$ is not strictly outer.  This contradiction proves the lemma.
\end{proof}

\begin{lem}\label{continuous_trace_c_a_proposition_galois}\cite{pedersen:ca_aut}
Let $G$ be a finite group and  $f: A^G \rar A$ is a $G$ - covering projection. If $A^G$ is a continuous trace $C^*$ - algebra then $A$ is also a continuous trace $C^*$ - algebra. 
\end{lem}
\begin{proof}
From \ref{irreducible_corr} it follows that for any irreducible representation $\rho : A \rar \mathcal{B}(H)$ there is a irreducible representation $\eta: A^G \rar B(H)$ such that
\begin{equation}\label{irr_rho_eta}
\rho|_{A^G}=\eta
\end{equation}
 Let $x\in A^G$ be an abelian element of $A^G$. From \ref{abelian_element_proposition} it follows that $\mathrm{dim} \ \eta(x) \le 1$ for any irreducible representation $\eta: A^G \rar \mathcal{B}(H)$. From (\ref{irr_rho_eta}) it follows that $\mathrm{dim} \ \rho(x) \le 1$ for any irreducible representation $\rho: A \rar \mathcal{B}(H)$. So any abelian element of $A^G$ is also an abelian element of $A$.
 Let $t \in \hat A$ and $s =\hat f(t) \in \hat A^G$ where $\hat f$ is defined by (\ref{spectrum_galois_map}). From \ref{abelian_element_proposition} it follows that there is an abelian element $x\in A^G$ such that $\hat x \in K(\hat A^G)$ and $\hat x(s)=1$. However $x$ is a abelian element of $A$, $\hat x \in K(A)$ and $\hat x(t)= \hat x(s)=1$. From \ref{continuous_trace_c_a_proposition} it follows that $A$ is a continuous trace $C^*$ - algebra.
 
\end{proof}

\begin{prop}\label{prop_act_grp}\cite{bourbaki_sp:gt}
If a topological group $G$ acts properly on a topological space then orbit space $X/G$ is Hausdorff. If also $G$ is Hausdorff, then $X$ is Hausdorff.
\end{prop}

\begin{thm}\label{finite_ctra_cov_pr_free ation}
 Let  $f:A^G \rar A$ be a noncommutative finite covering projection and $A^G$ is a continuous trace algebra. Then is a  $\hat A \rightarrow \hat A /G$ is a (topological) covering projection.
\end{thm}
\begin{proof}

From lemma \ref{continuous_trace_c_a_proposition_galois} it follows that $A$ is a continuous trace algebra. From  \ref{continuous_trace_c_a_proposition} it follows that a space $\hat A$ is Hausdorff. From \ref{free_action_ccr} it follows that $G$ acts freely on $\hat A $. From (\ref{spectrum_galois_quot}) it follows that $\widehat{A^G} \approx \hat A / G$.  It is known \cite{spanier:at} that if a finite group $G$ acts freely on Hausdorff space $X$ then $X \rightarrow X/G$ is a covering projection.
\end{proof}

\begin{rem} From  theorem \ref{finite_ctra_cov_pr_free ation} it follows that finite covering projections of commutative algebras are just covering projections of their character spaces.
If $A^G$ is a commutative $C^*$ - algebra then $\mathrm{dim} \ \pi(A^G)=1$ for all irreducible $\pi: A \rar \mathcal{B}(H)$. If $f: A^G \rar A$ is noncommutative $G$ covering projection and $A^G$ is commutative then $A^G$ is continuous trace algebra $\Omega(A^G) \approx \hat A^G$. From \ref{continuous_trace_c_a_proposition_galois} it follows that $A$ is also a continuous trace $C^*$ - algebra. If $\rho: A \rar \mathcal{B}(H)$ then $\rho (A)=\mathcal{K}(H)$. Let us recall construction from \ref{irreducible_corr}.
Let us enumerate elements of $G$ by integers, i. e. $g_1, ..., g_n \in G , \ n = |G|$ and define action of $\sigma: G \times \{i, ..., n\} \rar \{i, ..., n\}$ such that $\sigma(g, i) = j \Leftrightarrow  g_j = gg_i$
Let $\rho_{\oplus} = \oplus_{g \in G} \rho_g: A\rar \mathcal{B} (H^n)$ be such that
\begin{equation}\label{a_act_irr_sum1}
\rho_{\oplus} (a)(h_1, ..., h_n)= (\rho(g_1a)h_1, ...,  (\rho(g_na)h_n).
\end{equation}
Let us define such linear action of $G$ on $H^n$ that
\begin{equation}\label{g_act_irr_sum1}
g (h_1, ..., h_n)  = (h_{\sigma(g^{-1}, 1)}, ...,h_{\sigma(g^{-1}, n)}).
\end{equation}
From (\ref{a_act_irr_sum}), (\ref{g_act_irr_sum1}) it follows that 
\begin{equation}\nonumber
g(ah)=(ga)(gh); \ \forall a\in A, \ \forall g\in G, \ \forall h\in H^n,
\end{equation}
i.e. $H^n \in _A\mathcal{M}^G$. Representation $\rho_{\oplus}$ defines representation $\eta: A^G \rar \mathcal{B}(K)$. $K = \left(H^n\right)^G$. From \ref{irreducible_corr} $\eta$ is irreducible representation and since $A^G$ is commutative it follows that $\mathrm{dim} \ K=1$. From (\ref{g_act_irr_sum1}) it follows that $\mathrm{dim} \ H = 1$. Thus the dimension of any irreducible representation of $A$ equals to 1. It means that any irreducible representation is commutative. From this fact it follows that $A$ is a commutative $C^*$ - algebra $\hat A = \Omega(A)$ and $\Omega(f): \Omega(A) \rar \Omega(A^G)$ is a (topological) covering projection.

\end{rem}

\subsection{Covering projections of noncommutative torus}\label{cov_pr_nc_torus}
\begin{empt}
A noncommutative torus \cite{varilly:noncom} $A_{\theta}$ is $C^*$-norm completion of algebra generated by two unitary elements $u, v$ which satisfy following conditions
\begin{equation}\nonumber
uu^*=u^*u=vv^*=v^*v=1;\\
\end{equation}
\begin{equation}\nonumber
uv=e^{2\pi i \theta}vu,
\end{equation}
where $\theta \in \Rb$.
If $\theta = 0$ then $A_{\theta}=A_0$ is commutative algebra of continuous functions on commutative torus $C(S^1\times S^1)$.
There is a trace $\tau_0$ on  $A_{\theta}$ such that $\tau_0 (\sum_{-\infty < i < \infty, -\infty < j<\infty}a_{ij}u^iv^j) = a_{00}$. $C^*$ - norm of $A_{\theta}$  is defined by following way $\|a\|=\sqrt{\tau_0 (a^*a)}$. Let us consider * - homomorphism  $f:A_{\theta} \rightarrow A_{\theta'}$, where $A_{\theta'}$ is generated by unitary elements $u'$ and $v'$. Homomorphism $f$ is defined by following way:
\begin{equation}\nonumber
u \mapsto u'^m;
\end{equation}
\begin{equation}\nonumber
v \mapsto v'^n;
\end{equation}
It is clear  that
\begin{equation}\label{theta_k}
\theta' = \frac{\theta + k}{mn}; \ (k = 0,..., mn - 1).
\end{equation}
\end{empt}
\begin{lem}
Above $*$-homomorphism $A_{\theta}\rightarrow A_{\theta'}$ is a noncommutative covering projection.
\end{lem}
\begin{proof}

We need check conditions of definition \ref{nc_fin_cov_pr_defn}. $A_{\theta'}$ is a free $A_{\theta}$ module generated by monomials $u'^iv'^j$ ($i = 0,..., m-1; \ j = 0,..., n-1$), so it is projective finitely generated $A_{\theta}$-module. Commutative $C^*$- subalgebras $C(u') \subset A_{\theta'}$ and  $C(v') \subset A_{\theta'}$ generated by $u'$ and $v'$ respectively are isomorphic to algebra $C(S^1)$, where $S^1$ is one dimensional circle. There are induced by $f$ *-homomorphisms $C(S^1) = C(u) \rightarrow C(u') = C(S^1)$ ,
  $C(S^1) = C(v) \rightarrow C(v') = C(S^1)$. These *-homomorphisms induces $m$ and $n$ listed covering projections respectively. Covering groups of these covering projections are $G_1\approx\Zb_m$ and $G_2\approx\Zb_n$ respectively. Generators of these groups are presented below:
 \begin{equation}\label{ntorus_gen_1}
 u' \mapsto e^{\frac{2\pi i}{m}} u';
 \end{equation}
 \begin{equation}\label{ntorus_gen_2}
 v' \mapsto e^{\frac{2\pi i}{n}} v'.
 \end{equation}
 Equations (\ref{ntorus_gen_1}), (\ref{ntorus_gen_2}) define action of $G = \mathbb{Z}_m \times \mathbb{Z}_n$ on $A_{\theta'}$ and $A_{\theta}=A_{\theta'}^G$. Inner automorphisms of $A_{\theta'}$ are given by
 \begin{equation}\nonumber
 v' \mapsto u'^pv'u'^{*p} = e^{\frac{2\pi i p \theta}{mn}}v'.
 \end{equation} 
  \begin{equation}\nonumber
  u' \mapsto v'^qu'v'^{*q} = e^{\frac{2\pi i q \theta}{mn}}u'.
  \end{equation} 
  These inner automorphisms do not coincide with automorphisms given by (\ref{ntorus_gen_1}), (\ref{ntorus_gen_2}).
  Let us show that $\mathrm{can}: A_{\theta'} \otimes_{A_{\theta}}  A_{\theta'} \rar \mathrm{Map}(G, A_{\theta'})$ is an isomorphism in $_{A_{\theta}}\mathcal{M}^G$ category. This fact follows from the set theoretic bijectivity of the $\mathrm{can}$.
 Homomorphisms of commutative algebras $C(u)\rightarrow C(u')$, $C(v)\rightarrow C(v')$ correspond to covering projection, it follows that there are elements $x_i \in C(u')$ ($i=1,..., r$), $y_j\in C(v')$ ($j=1,..., s$) such that
 \begin{equation}\label{tor_eqn_1}
 \sum_{1 \le i \le r}{x^{2}_i} = 1_{C(u')};
 \end{equation}
 \begin{equation}\label{tor_eqn_2}
 \sum_{1 \le i \le r}{x_i(g_1 x_i)} = 0; g_1 \in G_1;
 \end{equation}
 \begin{equation}\label{tor_eqn_3}
 \sum_{1 \le j \le s}{y^{2}_i} = 1_{C(v')};
 \end{equation}
 \begin{equation}\label{tor_eqn_4}
 \sum_{1 \le j \le s}{y_i(g_2 y_i)} = 0; g_2 \in G_2,
 \end{equation}
 where $g_1$ and $g_2$ are nontrivial elements of $\mathbb{Z}_m$ and $\mathbb{Z}_n$.
 
 Let $a_k, b_k \in A_{\theta}$ be such that

 \begin{equation}\nonumber
 a_k=y_j x_i,
 \end{equation}
 \begin{equation}\nonumber
 b_k= x_i y_j,
 \end{equation}
 
 where $k=1, ..., rs$.
 
From (\ref{tor_eqn_1})- (\ref{tor_eqn_4}) it follows that
 \begin{equation}\nonumber
 \sum_{1 \le k \le rs}{a_k b_k} = 1_{A_{\theta '}};
 \end{equation}
 \begin{equation}\nonumber
 \sum_{1 \le k \le rs}{a_k(gb_k}) = 0,
 \end{equation}
where $g\in G=\mathbb{Z}_m \times \mathbb{Z}_n$ is a nontrivial element. If $\varphi \in \mathrm{Map}(G, A_{\theta})$ is such that $g_i \mapsto c_i$ ($i = 1, ..., mn$) then
\begin{equation}
\varphi = \mathrm{can}\left(\sum_{i=1}^{mn} \sum_{k=1}^{rs}a_k \otimes g_i^{-1}b_kc_i\right).
\end{equation}
So $\mathrm{can}$ is a surjective map. Let us show that $\mathrm{can}$ is injective. $A_{\theta'}$  is a free left $A_{\theta}$ module, because any element $a \in A_{\theta'}$ has following unique representation
\begin{equation}\label{free_mod_torus_eqn}
a = \sum_{r=0, s=0}^{m-1, \ n-1} a_{rs}u'^rv'^s \ (a_{rs} \in A_{\theta}).
\end{equation}
From (\ref{free_mod_torus_eqn}) it follows that any element $x \in A_{\theta'} \otimes_{A_{\theta}} A_{\theta'}$ has following unique representation
\begin{equation}\label{free_mod_torus_ten}
x = \sum_{r=0, s=0}^{m-1, \ n-1} a_{rs} \otimes u'^rv'^s \ (a_{rs} \in A_{\theta'}).
\end{equation}
Let us prove that $\mathrm{can}$ maps above sum of linearly independent elements of $A_{\theta'} \otimes_{A_{\theta}} A_{\theta'}$ to sum of linearly independent elements of $\mathrm{Map}(\mathbb{Z}_m \times \mathbb{Z}_n, A_{\theta'})$. Really if 
\begin{equation}\label{free_mod_torus_rep}
\varphi = \mathrm{can}(a \otimes u'^r v'^s)
\end{equation}
and $(p, q)\in \mathbb{Z}_m \times \mathbb{Z}_n$ then
\begin{equation}\label{free_mod_torus_phi}
\varphi((p,q)) = \varphi((0,0))e^{\frac{2\pi i pr}{m}}e^{\frac{2\pi i qs}{n}}.
\end{equation}
i.e. linearly independent elements of (\ref{free_mod_torus_ten}) correspond to different representations of $G = \mathbb{Z}_m \times \mathbb{Z}_n$, but different representations are linearly independent. So $\mathrm{can}$ is injective.

\end{proof}

\begin{rem}\label{rem_torus_morita}
Let  $\theta \in \mathbb{R}$ be irrational number, $m,n \in \mathbb{N}$, $mn>1$, $\theta' = \theta/mn$, $\theta'' = (\theta + k)/mn$ ($k \ne 0 \ \mathrm{mod} \ mn$). Let $u, v\in A_{\theta}$, $u', v'\in A_{\theta'}$, $u'', v''\in A_{\theta''}$ be unitary  generators, $f': A_{\theta} \rar A_{\theta'}$ (resp.  $f'': A_{\theta} \rar A_{\theta''}$) be * - homomorphism $u \mapsto u'^m$, $v \mapsto v'^n$ (resp. $u \mapsto u''^m$, $v \mapsto v''^n$). We have $A_{\theta'} \not\approx A_{\theta''}$. So this noncommutative covering projections are not isomorphic. However these covering projections can be regarded as equivalent because they are Motita equivalent. Let $U, V \in \mathbb{M}_{N=mn}(\mathbb{C})$ be unitary matrices such that
\begin{equation}\nonumber
UV = e^{\frac{2\pi i k }{ nm}}VU.
\end{equation}
There is following $G$ equivariant isomorphism $A_{\theta'} \otimes \mathbb{M}_N(\mathbb{C}) \approx A_{\theta''} \otimes \mathbb{M}_N(\mathbb{C})$
\begin{equation}\nonumber
u' \otimes 1 \rar u'' \otimes U; \  v' \otimes 1 \rar v'' \otimes V.
\end{equation}
This isomorphism is also $A_{\theta} - A_{\theta}$ bimodule isomorphism. From $\mathcal{K} \otimes \mathbb{M}_N(\mathbb{C}) \approx \mathcal{K}$ it follows that there exist isomorphism $A_{\theta'} \otimes \mathcal{K} \approx A_{\theta''} \otimes \mathcal{K}$ and there is following commutative diagram
\[
\begin{diagram}\label{group_dia_comm}
    \node{A_{\theta'} \otimes \mathcal{K}} \arrow[2]{e,t}{\approx} \node[2]{A_{\theta''} \otimes \mathcal{K}}  \\
   \node[2]{A_{\theta} \otimes \mathcal{K}}\arrow{nw}\arrow{ne}.
\end{diagram}
\]
I find that good theory of noncommutative covering projections should be invariant with respect to Morita equivalence. This theory can replace $C^*$-algebras with their stabilizations (recall that the stabilization of a $C^*$ algebra $A$ is a $C^*$-algebra $A\otimes \mathcal{K}$).
\end{rem}

\section{Covering projections and $K$-homology}
\subsection{Extensions of $C^*$-algebras generated by unitary elements}
\begin{defn}\label{gen_by_v_extension}
Let $A$ be a $C^*$-algebra, $A\rightarrow B(H)$ is a faithful representation, $u \in U(A^+)$, $v \in U(B(H))$,  is such that $v^n=u$ and $v^i \notin U(A^+)$, ($i=1,..., n-1$). A {\it generated by $v$ extension} is a minimal subalgebra of $B(H)$ which contains following operators: 
\begin{enumerate}
\item $v^i a; \ (a \in A, \ i=0, ..., n-1)$
\item $a v^i$.
\end{enumerate}
Denote by $A\{v\}$ a generated by $v$ extension.
\end{defn}
\begin{rem}
Sometimes a $*$-homomorphism $A \rightarrow A\{v\}$ is a noncommutative covering projection but it is not always true. If the homomorphism is a covering projection then there is a relationship between the covering projection and $K$ - homology.
\end{rem}
\begin{lem} Let $A$ be a $C^*$-algebra, $A\rightarrow B(H)$ is a faithful representation,  $u\in U(A^+)$ is an unitary element such that $\mathrm{sp}(u)=\mathbb{C}^*=\{z \in \mathbb{C} \ | \ |z| = 1 \}$, $\xi, \eta \in B_{\infty}(\mathrm{sp}(u))$ are Borel measured functions such that $\xi(z)^n = \eta(z)^n = z$ ($\forall z \in\mathrm{sp}(u)$). Then there is an isomorphism 
\begin{equation}
A\{\xi(u)\} \otimes \mathcal{K} \rightarrow A\{\eta(u)\} \otimes \mathcal{K} 
\end{equation}
which is a left $A$-module isomorphism. The isomorphism is given by
\begin{equation}
\xi(u) \otimes x \mapsto \eta(u) \otimes \xi\eta^{-1}(u)x; \ (x \in \mathcal{K}).
\end{equation}
\end{lem}
\begin{proof}
Follows from the equality $\xi(u)= \xi\eta^{-1}(\eta(u))$.
\end{proof}
\begin{rem}
See remark \ref{rem_torus_morita}.
\end{rem}
\begin{defn}\label{root_n_defn}
A {\it $n^{\mathrm{th}}$ root of identity map} is a Borel-measurable function $\phi \in \ B_{\infty}(\mathbb{C}^*)$  such that 
\begin{equation}\label{root_n_eqn}
(\phi(z))^n = z \ (\forall z \in U(C(X)). 
\end{equation}
 \end{defn}
 \begin{lem}\label{full_sp_lem}
 Let $A$ be a $C^*$-algebra, $u \in U((A \otimes \mathcal{K})^+)$ is such that $[u]\neq 0 \in K_1(A)$ then $\mathrm{sp}(u)= \mathbb{C}^*=\{z \in \mathbb{C} \ | \ |z| = 1\}$.
 \end{lem}
 \begin{proof}
  $\mathrm{sp}(u) \subset \mathbb{C}^*$ since $u$ is an unitary. Suppose $z_0 \in \mathbb{C}$ be such that  $z_0 \notin \mathrm{sp}(u)$ and $z_1 = -z_0$. Let $\varphi: \mathrm{sp}(u) \times [0,1] \to \mathbb{C}^*$ be such that
  \begin{equation*}
  \varphi(z_1e^{i \phi}, t)= z_1 e^{i (1-t)\phi}; \ \phi \in (-\pi, \pi), \ t \in [0,1].
  \end{equation*}
  There is a homotopy $u_t = \varphi(u, t)\in  U((A \otimes \mathcal{K})^+)$ such that $u_0 = u$, $u_1 = z_1$. From $[z_1] = 0 \in K_1(A)$ it follows that $[u]=0\in K_1(A)$. So there is a contradiction which proves this lemma.
  
 \end{proof}

\subsection{Universal coefficient theorem}

Universal coefficient theorem \cite{blackadar:ko} establishes (in particular) a relationship between $K$ - theory and $K$- homology. For any $C^*$-algebra $A$ there is a natural homomorphism
\begin{equation}\label{free_spec_eqn}
\gamma: KK_1(A, \mathbb{C}) \rar \mathrm{Hom}(K_1(A), K_0(\mathbb{C})) \approx \mathrm{Hom}(K_1(A), \mathbb{Z})
\end{equation}
which is the adjoint of following pairing
\begin{equation}\nonumber
KK(\mathbb{C}, A) \otimes KK(A, \mathbb{C}) \rar KK(\mathbb{C}, \mathbb{C}).
\end{equation}
If $\tau \in KK^1(A, \mathbb{C})$ is represented by extension
\begin{equation}\nonumber
0 \rar  \mathbb{C} \rar D \rar A \rar  0
\end{equation}
then $\gamma$ is given as connecting maps $\partial$ in the associated six-term exact sequence of $K$ theory

 \[
\begin{diagram}
 \node{ K_0(\mathbb{C})) } \arrow{e}  \node{K_0(D)}  \arrow{e}  \node{K_0(A)}  \arrow{s,r}{\partial}  \\
\node{K_1(A)} \arrow{n, l}{\partial} \node{K_1(D)} \arrow{w} \node{K_1(\mathbb{C})} \arrow{w}
\end{diagram}
\]
If $\gamma(\tau)=0$ for an extension $\tau$ then the six-term $K$-theory exact sequence degenerates into two short exact sequences 
\begin{equation}\nonumber
0 \rar K_i(A) \rar \K_i(D) \rar K_i(\mathbb{C}) \rar 0 \ (i=0,1)
\end{equation}
and thus determines an element $\kappa(\tau)\in \mathrm{Ext}^1(K_*(A), K_*(\mathbb{C})$. 
In result we have a sequence of abelian group homomorphisms
\begin{equation}\nonumber
\mathrm{Ext}^1(K_0(A), K_0(\mathbb{C})) \rar KK^1(A, \mathbb{C}) \rar \mathrm{Hom}(K_1(A), K_0(\mathbb{C}))
\end{equation}
such that composition of the homomorphisms is trivial. Above sequence can be rewritten by following way
\begin{equation}\label{uct_c}
\mathrm{Ext}^1(K_0(A), \mathbb{Z}) \rar K^1(A) \rar \mathrm{Hom}(K_1(A), \mathbb{Z})).
\end{equation}
If $G$ is an abelian group that 
\begin{equation}\nonumber
\mathrm{Ext}^1(G, \mathbb{Z}) = \mathrm{Ext}^1(G_{tors}, \mathbb{Z}),
\end{equation}
\begin{equation}\nonumber
 \mathrm{Hom}(G, \mathbb{Z}) =  \mathrm{Hom}(G / G_{tors}, \mathbb{Z})).
\end{equation}
From (\ref{uct_c}) it follows that $K^1(A)$ depends on $K_0(A)_{tors}$ and $K_1(A)/K_1(A)_{tors}$. We say that dependence(\ref{uct_c}) on $K_0(A)_{tors}$ is a {\it torsion special case} and dependence (\ref{free_spec_eqn}) of $K^1(A)$ on $K_1(A)/K_1(A)_{tors}$ is a {\it free special case}.

\subsection{Free special case}\label{free_spec_case}

\begin{exm}\label{circle_cov_proj}
The $n$- listed coverings of example \ref{circle_cov} can be constructed algebraically.  From (\ref{uct_c}) it follows that  $K_1(C(S^1))\approx \mathbb{Z}$. Let $u \in U(C(S^1))$ is such that $[u] \in K_1(S^1)$ is a generator of $K_1(S^1)$. Let $C(S^1)\rar B(H)$ be a faithful representation and $\phi$ is an $n^{\mathrm{th}}$ root of identity map. If $v = \phi(u) \in B(H)$ then $v^n = u$ and $v \notin C(S^1)$. According to definition \ref{gen_by_v_extension} we have a $*$ - homomorphism $C(S^1)\rightarrow C(S^1)\{v\}$ which corresponds to $n$ listed covering projection of the $S^1$.
\end{exm}
\begin{empt}\label{free_special_case_general} {\it General construction}.
Construction of example \ref{circle_cov_proj} can be generalized. Let $A$ be a $C^*$ - algebra such that $K^1(A)\approx G \oplus \mathbb{Z}$. From (\ref{uct_c}) it follows that
\begin{equation}\label{k1_direct_sum} 
K_1(A)=G' \oplus \mathbb{Z}[u]
\end{equation}
where $u \in U((A \otimes \mathcal{K})^+)$. If $\phi$ is an $n^{\mathrm{th}}$ - root of identity map then we have a generated by $\{\phi(u)\}$ extension $A\rightarrow A\{\phi(u)\}$. Sometimes this extension is a noncommutative covering projection.

\end{empt}
\begin{exm}\label{torus_uni_cov_sample}
Let $A_{\theta}$ be a noncommutative torus, $K_1(A_{\theta})\approx \mathbb{Z}^2$ Let $u, v \in U(A)$ be representatives of generators of $K^1(A_{\theta})$ a $\mathrm{sp}(u)=\mathrm{sp}(v)=\{z \in \mathbb{C} \ | \ |z| = 1 \}$. Following $*$-homomorphisms
\begin{equation}\nonumber
A_{\theta}\rightarrow A_{\theta}\{\phi(u)\},
\end{equation}
\begin{equation}\nonumber
A_{\theta}\rightarrow A_{\theta}\{\phi(v)\}
\end{equation}
are particular cases of noncommutative covering projections which are described in subsection \ref{cov_pr_nc_torus}.
 \end{exm}
\begin{exm}\label{s_3_covering_sample}
It is known that $S^3$ is homeomorphic to $SU(2)$, $K_1(C(SU(2))) \approx \mathbb{Z}$ and $K_1(C(SU(2)))$ is generated by unitary $u \in U(C(SU(2) \otimes \mathbb{M}_2(\mathbb{C}))$.  Element $u$ can be regarded as the natural map $SU(2) \rar \mathbb{M}_2(\mathbb{C})$ and $\mathrm{sp}(u)=\{z\in \mathbb{C}\ | \ |z|=1\}$. Denote by $A=C(SU(2)) \otimes \mathbb{M}_2(\mathbb{C})$. Let $\phi$ be a $2^{\mathrm{th}}$ - root of identity map, and $v = \phi(u)$. There is an extension $A\rightarrow A\{v\}$. Both $A$ and $A\{v\}$ are continuous trace algebras. The  $\mathbb{Z}_2$ group acts on $A\{v\}$ such that action of nontrivial element $g\in \mathbb{Z}_2$ is given by
\begin{equation}\nonumber
gv = -v.
\end{equation}
Let $\rho: A\{v\} \rightarrow B(H)$ be a irreducible representation. Then $V = \rho(v)$ is a $2 \times 2$ unitary matrix. Suppose that $\rho$ is such that
by
\begin{equation}\nonumber
\rho(v)=\begin{pmatrix}
1 & 0 \\
0 & -1 \\
\end{pmatrix}.
\end{equation}
We have
\begin{equation}\nonumber
\rho_g(v)=\rho(gv)\begin{pmatrix}
-1 & 0 \\
0 & 1 \\
\end{pmatrix}.
\end{equation}
Above matrices are unitary equivalent, i. e.
\begin{equation}\nonumber
\begin{pmatrix}
1 & 0 \\
0 & -1 \\
\end{pmatrix} = 
\begin{pmatrix}
0 & -1 \\
1 & 0 \\
\end{pmatrix}
\begin{pmatrix}
-1 & 0 \\
0 & 1 \\
\end{pmatrix}
\begin{pmatrix}
0 & 1 \\
-1 & 0 \\
\end{pmatrix}
\end{equation}.
  
So the representation $\rho$ is unitary equivalent to the $\rho_g$ and action of $g$ is not strictly outer,  extension $f: A \rightarrow A\{v\}$ does not satisfy definition \ref{nc_fin_cov_pr_defn}, i.e. $f: A \rightarrow A\{v\}$ is not a noncommutative covering projection. Algebra $A$ does not have nontrivial noncommutative covering projections because
\begin{enumerate}
\item $A$ is a continuous trace algebra,
\item $\hat A \approx S^3$,
\item $\pi_1(S^3)=0$, i.e. $S^3$ does not have nontrivial covering projections.
\end{enumerate}
\end{exm}
\begin{rem}\label{pi1_belongs}
This construction supplies a covering projection if $x\in K_1(X)$ belongs to image of $\pi_1(X) \rightarrow K_1(X)$.
\end{rem}

\subsection{Torsion special case}\label{torsion_special_case}

\begin{exm}\label{n_cov_sample}
Universal covering from example \ref{n_cirle_cov} can be constructed algebraically.
Let $f: S^1 \rar S^1$ be a $n$ listed covering projection  of the circle, $C_f$ is the (topological) mapping cone of $f$. $C(f): C(S^1) \rar C(S^1)$ is a corresponding *- homomorphism of $C^*$-algebras ($u \mapsto u^n$),
where $u \in U(C(S^1))$ is such that $[u]\in K_1(C(S^1))$ is a generator. Algebraic mapping cone \cite{blackadar:ko}  $C_{C(f)}$  of $C(f)$ corresponds to the topological space $C_f$. $C_{C(f)}$ is an algebra of continuous maps $f [0, 1)\rar U(\mathbb{C})$ such that

\begin{equation}\nonumber
f(0) = \sum_{k \in \mathbb{Z}} a_k u^{kn}, \ a_k \in \mathbb{C}.
\end{equation}
A map $v = (x \mapsto  u)$ ($\forall x \in [0, 1]$) is such that $v^i\notin M(C(C_f))$ ($i=1,...,n-1$), $v^n \in M(C_{C(f)})$. Homomorphism $C_{C(f)}\rar C_{C(f)}\{v\}$ corresponds to a $n$-listed covering projection from the example \ref{n_cirle_cov}.

\end{exm}
\begin{empt}\label{tors_special_case_general}{\it General construction}.
Above construction can be generalized. Let $A$ be a $C^*$ - algebra such that $K^1(A)= G \oplus \mathbb{Z}_n$, where $G$ is an abelian group. From (\ref{uct_c}) it follows that $K_0(A) \approx G' \oplus \mathbb{Z}_n$.
Let $Q^s(A)= M(A \otimes  \mathcal{K})/ (A \otimes \mathcal{K}$) be the stable multiplier algebra of $C^*$ - algebra $A$. Then from \cite{blackadar:ko} it follows that $K_1(Q^s(A))=K_0(A)$. Let $u\in U(Q^s(A))$ be such that $K_1(Q^s(A))=G' \oplus \mathbb{Z}_n[u]$. Let $\phi$ be a $n^{\mathrm{th}}$ root of identity map such that $\phi(u^n)=u$. Let $p:  M(A \otimes  \mathcal{K}) \rar  M(A \otimes  \mathcal{K}) / (A \otimes  \mathcal{K})$ be a natural surjective *- homomorphism. It is known \cite{blackadar:ko} that unitary element $v \in  U(Q^s)$ can be lifted to an unitary element $v'\in U(M(A \otimes \mathcal{K}))$ (i.e. $v = p(v')$) if and only if $[v]=0\in K^1(Q^s(A))$.  From $n[u]=[u^n]=0$ it follows that there is an unitary $w \in U(M(A \otimes \mathcal{K})$) such that $p(w)=u^n$. Let $M(A \otimes  \mathcal{K})\rightarrow B(H)$ be a faithful representation, then $\phi(w)\in U(B(H))$. If $\phi(w)\in M(A \otimes  \mathcal{K}))$ then $p(\phi(w)) = u$, however it is impossible because $[u] \neq 0 \in K^1(Q^s(A))$. So $\phi(w) \notin M(A \otimes  \mathcal{K})$ and similarly  $\phi(w)^i \notin M(A \otimes  \mathcal{K})$ ($i = 1,..., n -1)$. So we have a generated by $\phi(w)$ extension $A\otimes \mathcal{K} \rightarrow (A \otimes \mathcal{K})\{\phi(w)\}$ which can be a noncommutative covering projection. Example \ref{n_cov_sample} is a particular case of this general construction.

\end{empt}
\begin{exm}
Let $O_n$ be a Cuntz algebra \cite{blackadar:ko}, $K_0(O_n)=\mathbb{Z}_{n-1}$. Construction \ref{tors_special_case_general} supplies a $\mathbb{Z}_{n-1}$ - Galois extension $f:O_n \otimes \mathcal{K} \to \widetilde{O}_n$. However it is not known is $f$ strictly outer.
\end{exm}
\subsection{A noncommutative generalization of $K_{11}(X)$.}
Above construction can generalize $K_{11}(X)$ group. Suppose that $K_1(X)$ is group generated by $x_1, ..., x_n$. Let $x\in \{x_1,...,x_n\}$ be a generator.  Construction of  \ref{free_special_case_general}, \ref{tors_special_case_general}
supplies extension of $A$ which is associated with $x$. The element $x$ is said to be {\it proper} if the extension is a noncommutative covering projection. Generalization of $K_{11}(X)$ is a generated by proper elements subgroup of $K^1(A)$.

\section{Conclusion}
The presented here theory supplies algebraic construction of covering projections. These projections are well known for commutative case. Example \ref{torus_uni_cov_sample} is principally new application of the theory. It is interesting to find other nontrivial examples of this theory.


\begin{thebibliography}{10}

\bibitem{blackadar:ko}
B. Blackadar. {\it K-theory for Operator Algebras}, Second edition.
Cambridge University Press 1998.

\bibitem{bourbaki_sp:gt}
N. Bourbaki, {\it General Topology}, Chapters 1-4, Springer, Sep 18, 1998. 


\bibitem{morita_hopf_galois}
S. Caenepeel, S. Crivei, A. Marcus, M. Takeuchi. {\it Morita equivalences induced by bimodules over Hopf-Galois extensions}, arXiv:math/0608572, 2007.

\bibitem{dixmier_a_r} Jacques Dixmier. {\it Les C*-alg\`{e}bres et leurs repr\'esentations} 2e \'ed. Gauthier-Villars in Paris 1969.

\bibitem{dixmier_tr}J.Dixmier. {\it Traces sur les $C^*$-algebras}. Ann. Inst. Fourier, 13, 1(1963), 219-262, 1963

\bibitem{gross_gauge}David J. Gross. {\it
Gauge Theory-Past, Present, and Future?} Joseph Henry Luborutoties, Ainceton University, Princeton, NJ 08544, USA.
(Received November 3,1992).



\bibitem{miyashita:finite_outer_galois}
Y. Miyashita. {\it Finite outer Galois theory of non-commutative rings}, J. Fac. Sci. Hokkaido Univ. (I) 19 (1966), 114-134.

\bibitem{hajac:toknotes}
{\it Lecture notes on noncommutative geometry and quantum groups}, Edited by Piotr M. Hajac.


\bibitem{murphy}
G.J. Murhpy. {\it $C^*$-Algebras and Operator Theory.}
Academic Press 1990.


\bibitem{pedersen:ca_aut}Pedersen G.K. {\it $C^*$-algebras and their automorphism groups}. London ; New York : Academic Press, 1979.

\bibitem{rieffel_finite_g}
Marc A. Reiffel, {\it Actions of Finite Groups on $C^*$ - Algebras}. 	Department of Mathematics University of California Berkeley. Cal. 94720 U.S.A. 1980.



\bibitem{spanier:at}
E.H. Spanier. {\it Algebraic Topology.} McGraw-Hill. New York 1966.


\bibitem{varilly:noncom}
J.C. V\'arilly. {\it An Introduction to Noncommutative Geometry}. EMS 2006.



\end{thebibliography}
\end{document}